\documentclass[12pt]{article}
\usepackage{amsmath, amssymb, amsfonts, amsthm}
\usepackage{authblk}
\usepackage{enumitem}
\usepackage{multicol}
\usepackage{hyperref}
\usepackage{tikz}
\usepackage{array}
\usepackage{tabularx}
\tikzstyle{vertex}=[circle, draw, inner sep=0pt, minimum size=6pt]

\allowdisplaybreaks
\def \Z {{\mathbb{Z}}}

\newtheorem*{theorem*}{Theorem}
\newtheorem{theorem}{Theorem}
\newtheorem{cor}[theorem]{Corollary}

\newtheorem{lemma}[theorem]{Lemma}
\newtheorem{ex}[theorem]{Example}

\newtheorem{definition}[theorem]{Definition}

\newtheorem*{ex*}{Example}

\newtheorem{pro}[theorem]{Proposition}
\title{$C$-Width of Graph of Groups}

\author[]{Shrinit Singh}
\affil[]{Shiv Nadar Institution of Eminence, ss101@snu.edu.in}
\date{}

\begin{document}
\maketitle
\begin{abstract}
In this paper, we study the $C$-width of HNN extension of a group via its proper isomorphic subgroups and amalgamated free product of two groups via their proper isomorphic subgroups with respect to conjugation invariant generating set. We will also establish that infinite one relator group has infinite C-width.
\end{abstract}
{\bf{AMS(2020)}}:20F65, 20E06 \\

 \section{Introduction}
 Let $G$ be a group and $S$ be a generating set of $G$. To measure the ``length" of an element $g\in G$ with respect to $S$, denoted as $l_S(g)$, we count the minimum number of elements from $S \cup S^{-1}$ required to express $g$ as a product. The width of the group $G$ with respect to the generating set $S$, denoted as $wid(G, S)$, is defined as the supremum value of $l_S(g)$ for all elements $g \in G$.

Various authors have introduced and investigated various types of group width. Some of these width variants include verbal width, palindromic width, and $C$-width. For a comprehensive overview of results related to these different width concepts, readers can refer to the following references: \cite{faiziev2001problem, bardakov1997width, bardakov2006palindromic, BARDAKOV2005574, bardakov2012generating, bardakov2014palindromic, bardakov2015palindromic, bardakov2017palindromic, dobrynina2000width, dobrynina2009solution, gongopadhyay2020palindromic, rhemtulla_1968}. We will focus on $C$-width of some free constructions of groups.

 Before delving into $C$-width, we first define the free constructions: HNN extension and Amalgamated free product.
 
  {\bf HNN Extensions:} Let $G$ be a group and $A$ and $B$ be proper isomorphic subgroups of $G$ with isomorphism $\phi: A \longrightarrow B.$ Then the HNN extension of $G$ is $$G_* = \langle G,t | t^{-1}at = \phi(a),a \in A \rangle.$$

  {\bf Amalgamated free product:} Let $G_1$ and $G_2$ be two groups with subgroups $H_1$ and $H_2$ respectively such that $\phi : H_1 \longrightarrow H_2$ is an isomorphism. Then the free product of $G_1$ and $G_2$ amalgamating the subgroups $H_1$ and $H_2$ by the isomorphism $\phi$ is the group $$G= \langle G_1, G_2 | \phi(h) = h, h \in H_1 \rangle.$$ 
  Since $H_1$ and $H_2$ are identified in $G$, we will denote $H_1$ and $H_2$ by $H.$

   Bardakov et al \cite{bardakov2012generating} introduced the concept of finite $C$-width and defined it as follows:
  
\begin{definition}
    A group $G$ is said to have finite $C$-width if $G$ has finite width with respect to any conjugation invariant generating set $S.$ Otherwise, we say that $G$ has infinite $C$-width.
\end{definition}

Bardakov et al \cite{bardakov2012generating} proved the following result:

\begin{theorem}{\cite{bardakov2012generating}}
    Let $K$ be a group that is a free product of two nontrivial groups $K_1$ and $K_2$ . Then $K$ has finite $C$-width if and only if both $K_1$ and $K_2$ are group of order $2.$
\end{theorem}

 Moreover, they established another important result:
\begin{theorem}{\label{ext}\cite{bardakov2012generating}}
    Let $G$ be a group and $H$ be a normal subgroup of $G$ such that $H$ and $G/H$ both have finite $C$-width. Then $G$ also has finite $C$-width.
\end{theorem}

This theorem demonstrates that the property of finite $C$-width is closed under extensions of groups that possess finite $C$-width. However, it's important to note that finite $C$-width is not necessarily closed under subgroups, as exemplified by $\Z_2*\Z_2$, a group with finite $C$-width but a normal subgroup exhibiting infinite $C$-width. But we will show that the property of finite $C$-width is closed under taking quotients.

\begin{pro}{\label{finite c}}
    Let $G$ be a group which has finite $C$-width. Then each of its quotients has finite $C$-width.
\end{pro}

\begin{proof}
    On the contrary, let's assume that $G/N$ has infinite $C$-width with respect to a conjugation invariant set $S$. So we get a sequence of elements $(g_iN)$ in $G/N$ such that width of element $(g_iN)$ with respect to $S$ is strictly bigger than $ i.$

    Consider the conjugation invariant subset $S'$ of $G$ by writing all the elements of all the cosets in $S$. Then $S'$ is conjugate invariant.

    Now take the sequence of elements $(g_i)$ in $G$. Then we claim that $l_{S'}(g_i) > i$ and this implies $G$ has infinite $C$-width. If $l_{S'}(g_i) \leq i,$ then $g_i = h_1 \ldots h_r,$ where $r \leq i$ and this implies $g_iN = h_1N \ldots h_rN$ and this implies $l_S(g_iN) = r \leq i,$ which is a contradiction. And hence, each quotient has finite $C$-width.  
\end{proof}
    Here, we extended the result of \cite{bardakov2012generating} to the case of HNN extension and amalgamated free product. But before stating the result, lets compare verbal width and C-width.
    
    When considering words, it is evident that the images of these word maps are closed under automorphisms, and consequently, the images of word maps remain invariant under conjugation. While $C$-width and verbal width are very close but distinct concepts, the following observation is noteworthy: if a group exhibits infinite verbal width with respect to a set of words $S$, and verbal subgroup corresponding to the set $S$ is the whole group, then the group itself possesses infinite $C$-width. But the infinite cyclic group $\Z$ possesses finite verbal width for every word but infinite $C$-width.

    When a verbal subgroup with respect to set of word $S$ has finite width in free group, then  verbal subgroup with respect to the set of word $S$ has finite width in every group.  
    But the following two examples are interesting because the first example says that when verbal subgroup is whole group in case of free group, then width is finite and second example illustrates an already known example of group where verbal subgroup is whole group but width is infinite. 
    \begin{ex}
        Consider a subset $S$ of the free group $F_r$, where the image of the word map associated to $S$ in $F_r$ is denoted as $S(F_r)$. If the subgroup generated by $S(F_r)$ equals $F_r$, then the verbal width is finite. To see this result, let $x$ be a primitive element in the free group $F_r$, then length of $x$ with respect to $S(F_r)$ is finite, say $n$. Then any element $w \in F_r$ can be seen as an endomorphic image of $x$ and hence $w$ also has length $n$.
   \end{ex}

    \begin{ex}
        Let $G$ be an infinite simple group with infinite commutator width. The existence of such groups has been given by Alexey Muranov \cite{muranov2007finitely}. In this case, it is evident that such group has infinite $C$-width.
    \end{ex}

    In \cite{bardakov1997width}, Bardakov proved that one relator group with at least three generators has infinite verbal width. He further posed the question whether the verbal width for one relator group with two generators having a copy of non abelian free group is infinite or not. The next corollary is about $C$-width for one relator group.

    \begin{cor}
        Every one relator group with at least two generators have infinite $C$-width.
    \end{cor}

    \begin{proof}
         Utilizing Brodskii's result \cite{brodskii1980equations}, which states that every torsion-free one-relator group is homomorphic to an infinite cyclic group, and proposition \ref{finite c}, it is obvious that the corollary holds for one relator torsion free group. Every one relator group is a quotient of torsion free one relator group. Since torsion free one relator group has infinite $C$-width, hence every one relator group has infinite $C$-width by proposition \ref{finite c}.
    \end{proof}

Our main theorem are as follows:   

\begin{theorem}{\label{HNN}}
    Let $G_*$ be HNN extension of $G$ as defined above. Then $G_*$ has infinite $C$-width.
\end{theorem}    
\begin{theorem}{\label{amal}}
    Let $G$ be amalgamated free product as defined above. Then
    \begin{enumerate}
        \item{\label{amal1}} $G$ has infinite $C$-width if $|G_1:H| \geq 3$ and $|G_2:H| \geq 2.$
        \item{\label{amal2}} $G$ has finite $C$-width if  $|G_1:H| \leq 2$ and $|G_2:H| \leq 2$ and $H$ has finite $C$-width.
    \end{enumerate}
\end{theorem}

The above two theorems can be used to determine infinteness of $C$-width of graph of groups. To begin, we will give a concise introduction to graph of groups. Let $(G,X)$ be a graph of groups, by this we mean there exists a connected graph $X$, for each vertex $x$ and edge $e$ we associate it with $G_x$ and $G_e$ respectively such that there exists monomorphisms $\phi_{(0,e)}$ and $\phi_{(1,e)}$ taking the group associated to $e$, i.e. $G_e$ into the group associated to initial and terminal vertices of the edge $e$. We assume that $G_e = G_{\hat{e}}.$ For each vertex $v \in X,$ let $S_v$ be the generating set of $G_v$ and let $T$ be a maximal subtree in $x$. We fix $S = \{ \cup_{v \in vert X} S_v \} \cup \{ {edge{X} \setminus edge{T}}\}$ to be the standard generating set of fundamental group of graph of groups $(G,X), \pi_1(G,X).$ It is easy to see that the fundamental group of any graph of groups can be represented as amalgamated free product or HNN extension. Hence we have following corollary:

\begin{cor}
    Let $X$ be a non empty connected graph. Let $\pi_1(G,X)$ be the fundamental group of graph of groups of $X$ withh the standard generating set $S.$ Then the $C$-width of $\pi_1(G,X)$ is infinite if
    \begin{enumerate}
        \item \label{cor1} $X$ is a loop with a vertex $v$ and an edge $e$ such that $G_e$ is a proper subgroup of $G_v$; or
        \item \label{cor2}$X$ is a tree and has an oriented edge $e = [v_1, v_2]$ such that removing $e$, while retaining $v_1$ and $v_2$, gives two disjoint graphs $X_1$ and $X_2$ with $P_i \in vert X_i$ satisfying the following: extending $G_e \longrightarrow G_{P_i}$ to $\phi_i: G_e \longrightarrow \pi_1(G,  X_i), i = 1, 2,$ we get $[\pi_1(G, X_1) : \phi_1(G_e)] \geq 3$ and $[\pi_1(G, X_2) : \phi_2(G_e)] \geq 2.$
        \item \label{cor3} $X$ has an oriented edge $e = [v_1, v_2]$ such that removing the edge, while retaining $v_1$ and $v_2$ does not separate $x$ and gives a new graph $X'$ satisfying the following: extending $G_e \longrightarrow G_{v_i}$ to $\phi_i : G_e \longrightarrow \pi_1(G, X'), i = 1, 2,$ we have $\phi_i(G_e) = H_i$ and $H_1, H_2$ are proper subgroups of $\pi_1(G, X').$
    \end{enumerate}
\end{cor}

In first case \ref{cor1}, the fundamental group is an HNN extension of $G_v$ and hence the result follows from the theorem $\ref{HNN}.$ In second case, $\pi_1(G,X)$ is a free product of $\pi_1(G,X_1)$ and $\pi_1(G,X_2)$ amalgamating the subgroup $G_e$, then with the condition in the theorem and by theorem \ref{amal} part \ref{amal1}, the result follows. The final result follows from theorem \ref{HNN} as fundamental group in this case is an HNN extension of $\pi_1(G,X').$

 \section{HNN Extension}
We will give a proof of theorem \ref{HNN}.
\begin{proof}
     Let's consider the homomorphism from $G_*$ to $\Z$ via the extending the map $g \mapsto 1,$ when $g \in G$ and $t^n \mapsto n.$ Then we can see $\Z$ as homomorphic image of $G_*$. And since $\Z$ has infinite $C$-width, $G_*$ also has infinite $C$-width by proposition \ref{finite c}.
\end{proof}

\section{ Amalgamated free product}

    \subsection{Proof of theorem \ref{amal}, Part \ref{amal1}:} We employ the same methodology as Gongopadhyay and Krishna \cite{gongopadhyay2020palindromic}.  The proof is divided into two cases:
    \begin{itemize}
        \item \textbf{Case 1:} When there exists an element $a \in G_1 \cup G_2$ such that $HaH \neq Ha^{-1}H.$,
        \item \textbf{Case 2:}  When there does not exist any element $a \in G_1 \cup G_2$ such that $HaH \neq Ha^{-1}H.$
    \end{itemize}
    
    First we will prove the first case. Second case is just a small manipulation of the first case as done in \cite{dobrynina2009solution}.

\subsubsection{Quasi-homomorphisms:}
 A reduced sequence is a sequence denoted by $x_1, x_2, \ldots ,x_n$ where $n\geq 0$ and it exhibits the following properties:
 \begin{enumerate}
     \item Each $x_i$ is in one of the factors.
     \item Successive $x_i, x_{i+1}$ are not in same factor.
     \item If $n > 1$, no $x_i $ is in $H$.
     \item If $n = 1, x_1 \neq 1.$
\end{enumerate}
If $x_1,x_2, \ldots,x_n, n \geq 1$ is a reduced sequence, then the product $x_1\ldots x_n $ is not trivial in $G$ and we call it a reduced word. This representation is not unique but we have following lemma:

\begin{lemma}[\cite{gongopadhyay2020palindromic}]
    Let $g = x_1\ldots x_m$ and $h = y_1\ldots y_n$ be reduced words such that $g = h$ in $G$. Then $m=n.$
\end{lemma}

\begin{definition}
Let $g = x_1\ldots x_m$ be a reduced word in $G$. The elements $x_i$ are said to be syllables of $g$ and length of $g$, denoted by $l(g),$ is the number of syllables in $g.$ Here $l(g) = m.$
\end{definition}
\begin{definition}
    Let $a \in G_1 \cup G_2$ such that $HaH \neq Ha^{-1}H.$ Let $g \in G$ and $g = x_1\ldots x_m$ be a reduced word representing it. Then we define the special form of $g$ associated to this reduced word by replacing $x_i$ by $ua^{\theta}u',$ whenever $x_i = ua^{\theta}u'$ for some $u,u' \in H$ and $\theta \in\{ +1,-1 \}$, in the following way:
\begin{itemize}
    \item when $i =1, x_1 = ua^{\theta}u',$ we write $g = ua^{\theta }{x_2}'\ldots x_m,$ where ${x_2}' = u'x_2.$
    \item when $2 \leq i \leq n-1, x_i = ua^{\theta}u',$ we write $g = x_1x_2\ldots x'_{i-1}a^{\theta} {x'}_{i+1} \ldots x_m,$ where ${x'}_{i-1} = x_{i-1}u$ and ${x'}_{i+1} = u'x_{i+1}.$
    \item when $i = m$ and $x_m = ua^{\theta}u',$ where $\theta \in \{ +1, -1 \}$ and $u,u' \in H,$ we replace $x_n$ by $g = x_1x_2 \ldots x'_{m-1}a^{\theta}u',$ where ${x'}_{m-1} = x_{m-1}u.$
\end{itemize} 
\end{definition}

An $a$-segment of length $2k-1$ is a segment of the reduced word of the form $ax_1\ldots x_{2k-1}a,$ where $x_j \neq a$ for $j = 1, \ldots, 2k-1$ such that the length of the segment $x_1\ldots x_{2k-1}$ is $2k-1.$

Similarly an $a^{-1}$-segment of length $2k-1$ is a segment of the reduced word of the form $a^{-1}x_1\ldots x_{2k-1}a^{-1},$ where $x_j \neq a^{-1}$ for $j = 1, \ldots, 2k-1$ such that the length of the segment $x_1\ldots x_{2k-1}$ is $2k-1.$

For $g \in G$ expressed in a special form, we define
\[
\begin{aligned}
    p_k(g) &= \text{number of $a$-segments of length $2k-1,$}\\
    m_k(g) &= \text{number of $a^{-1}$-segments of length $2k-1,$}\\
    d_k(g) &= p_k(g) - m_k(g),\\
    r_k(g) &= \text{remainder of $d_k(g)$ divided by $2,$ and}
\end{aligned}
\]

\[
f(g) = \sum_{k=1}^{\infty} r_k(g).
\]

We have, from the definition, $p_k(g) = m_k(g^{-1})$ and hence $d_k(g^{-1}) + d_k(g) = 0$ for all $g \in G.$

\begin{lemma}[\cite{gongopadhyay2020palindromic}]
$f$ is well-defined on special forms.    
\end{lemma}

\begin{lemma}[\label{pro} \cite{gongopadhyay2020palindromic}]
    $f$ is a quasi-homomorphism, in fact, for any $g,h \in G, f(gh) = f(g) + f(h) + 9$ 
\end{lemma}

We will take $S = (G_1)^G \cup (G_2)^G$ to be conjugation invariant generating set of $G.$

\begin{lemma}\label{S}
    For any $s \in S$, we have $f(s) \leq 3.$
\end{lemma}
\begin{proof}
    Let $g^{-1}kg$ be an arbitrary element of the set $S$, where $g \in G$ and $k \in G_1 \cup G_2$. Suppose $g = x_1x_2\ldots x_m$ is an arbitrary reduced element of $G$. After reducing the word  $g^{-1}kg = x_m^{-1}\ldots x_1^{-1}kx_1\ldots x_m$, we get a reduced word $g'^{-1}k'g' = x_m^{-1}\ldots x_i^{-1}k'x_i\ldots x_m$ for some $i < m$ such that $x_i$ and $k'$ lie in different $G_i$. So, withouth loss of generality, we assume that the word  $g^{-1}kg = x_m^{-1}\ldots x_1^{-1}kx_1\ldots x_m$ is reduced word. We will consider two cases:

Case 1: $k \in (G_1 \cup G_2) \setminus \{a, a^{-1}\}$

Number of $a$-segment ($a^{-1}$-segment)in $g$ is same as number of  $a^{-1}$-segment ($a$-segment) in $g^{-1}$. If there is an $a$-segment containing $k$ if and only if there is an $a^{-1}$-segment of same length containing $k$.  
In this case, number of $a$-segment is the same as number of $a^{-1}$ segment. Hence $f(g^{-1}kg) = 0$.

Case 2: $k \in \{a, a^{-1}\}$

Number of $a$-segment ($a^{-1}$-segment)in $g$ is same as number of  $a^{-1}$-segment ($a$-segment) in $g^{-1}$. Suppose, if $k = a$, then we may get two extra $a$-segment of different length and one extra $a^{-1}$-segment different from the length of those two extra $a$-segment. We got $f(g^{-1}kg) \leq 3$. Similarly for $k = a^{-1}$, we get $f(g^{-1}kg) \leq 3$.
In each of these cases, we have shown that $f(g^{-1}kg) \leq 3$.
\end{proof}
\begin{lemma}\label{quasi}
    Let $g \in G$ be a product of $k$ many elements of $S$, then $f(s) \leq 12k-9.$
\end{lemma}
\begin{proof}
    Let $g = g_1\ldots g_k$, then by applying lemma \ref{pro}  iteratively, and applying lemma \ref{S}, we get this result.
\end{proof}
\begin{proof}
[\textbf{Proof of theorem \ref{amal}, part \ref{amal1}, case 1:}]  We will construct the same sequence as constructed in \cite{gongopadhyay2020palindromic} and show that the function $f$ is unbounded.

Let $b \in G_1 \setminus H$.

Let $g_1 = baba^{-1}ba$. Then $p_1(g_1) = 0$, $p_2(g_1) = 1$, and $p_k(g_1) = 0$ for $k \geq 2$. Also, $m_k(g_1) = 0$ for all $k$, $d_2(g_1) = 1$, and $d_k(g_1) = 0$ for all other $k$. So $f(g_1) = 1$.

Let $g_2 =  baba^{-1}baba^{-1}ba^{-1}ba$ and we get $f(g_2) = 2$.

For $g_3 = baba^{-1}baba^{-1}ba^{-1}baba^{-1}ba^{-1}ba^{-1}ba$, we get $f(g_3) = 4$. 

In general, for $g_n = baba^{-1}ba(ba^{-1})^2 \ldots ba(ba^{-1})^{n-1}ba(ba^{-1})^nba$, we get $f(g_n) \geq n-1$. By Lemma \ref{quasi}, we get that for any $k \in \mathbb{N}$, we have $g \in G$ such that $g$ cannot be written as a product of $n$ elements from $S$, where $n < k$. Hence, $G$ has an infinite $C$-width.

\end{proof}

\subsubsection{Case 2:}
For a non trivial $a \in G_1 \cup G_2$ such that $HaH = Ha^{-1}H.$ This case can be done using similar method as above. There will be just slight modification which can be done using \cite{dobrynina2000width} and \cite{dobrynina2009solution}.
\\

\noindent
\subsection{Proof of theorem 10, Part 2:}
\begin{proof}

     It is enough to assume that $|G_i:H_i| = 2.$ From theorem \ref{ext}, $G_1$ and $G_2$ have finite $C$-width as $H$ is a subgroup of index $2$ in $G_i$ and $H$ has finite $C$-width. $H$ is normal in $G$ as it is normal in $G_1$ and $G_2$ both. By universal property of amalgamated free product, $G/H$ is isomorphic to  ${\mathbb{Z}}_2 \star {\mathbb{Z}}_2$. Hence $G/H$ is also of finite $C$-width by \cite{bardakov2012generating}. And since finite $C$-width is closed under extension of group, we have that $G$, with a normal subgroup $H$ such that $H$ and $G/H$ both have finite $C$-width, is of finite $C$-width.
\end{proof}
\section*{Acknowledgement}
I want to thank Prof. Krishnendu Gongopadhyay for suggesting this problem during my visit to IISER Mohali.
\bibliographystyle{siam}
\bibliography{Ref}
\end{document}